\documentclass[12pt, twoside]{article}
\usepackage{latexsym}
\usepackage{amsmath}
\usepackage{amssymb}
\usepackage[all]{xy}
\usepackage{amsfonts}
\usepackage{verbatim}
\usepackage{amsthm}
\usepackage{mathrsfs}
\usepackage{epsfig}
\usepackage{xy}
\usepackage{array}
\usepackage{stmaryrd}
\usepackage{graphicx,color}
\usepackage{xcolor}
\usepackage{tikz}
\usetikzlibrary{arrows,calc}
\usepackage{etex}
\usepackage{mathdots}
\usepackage{float}
\usepackage{graphics}
\usepackage{pdflscape}
\usepackage{mathrsfs}

\usepackage{anysize}
\usepackage[colorlinks=true,linkcolor=black,citecolor=blue,urlcolor=black]{hyperref}
\input xypic
\xyoption{all}

\usepackage[perpage,symbol]{footmisc}
\usepackage{setspace}
\def\C{\mathscr{C}}\def\N{\mathscr{N}}

\def\dr{\ar@{->}[r]}

\numberwithin{equation}{section}

\begin{document}
\baselineskip=15pt
\title{\Large{\bf A pre-$(n+2)$-angulated category which is not\\[2mm] $(n+2)$-angulated}}
\medskip
\author{Jian He, Jing He and Panyue Zhou}
{\footnotetext{Jian He is supported by the National Natural Science Foundation of China (Grant Nos. 12501048, 12661005) and the Hongliu Outstanding Young Talents Funding of Lanzhou University of Technology. Jing He is supported by the National Natural Science Foundation of
China (Grant No. 12401045). Panyue Zhou is supported by the National Natural Science Foundation of China (Grant No. 12371034).}}

\date{}

\maketitle
\def\blue{\color{blue}}
\def\red{\color{red}}

\newtheorem{theorem}{Theorem}[section]
\newtheorem{lemma}[theorem]{Lemma}
\newtheorem{corollary}[theorem]{Corollary}
\newtheorem{proposition}[theorem]{Proposition}
\newtheorem{conjecture}{Conjecture}
\theoremstyle{definition}
\newtheorem{definition}[theorem]{Definition}
\newtheorem{question}[theorem]{Question}
\newtheorem{remark}[theorem]{Remark}
\newtheorem{remark*}[]{Remark}
\newtheorem{example}[theorem]{Example}
\newtheorem{example*}[]{Example}
\newtheorem{condition}[theorem]{Condition}
\newtheorem{condition*}[]{Condition}
\newtheorem{construction}[theorem]{Construction}
\newtheorem{construction*}[]{Construction}

\newtheorem{assumption}[theorem]{Assumption}
\newtheorem{assumption*}[]{Assumption}

\baselineskip=17pt
\parindent=0.5cm

\begin{abstract}
\baselineskip=16pt
We construct an explicit pre-$9$-angulated category which is not $9$-angulated, thereby giving a genuinely higher counterexample to the implication from pre-$(n+2)$-angulated to $(n+2)$-angulated. The underlying additive category is the category of finitely generated projective right modules over the preprojective algebra $\Pi(A_5)$ over $\mathbb F_2$. The construction is obtained by taking an odd power of the twisted complete comparison used by Chen--Liu--Lu--Zhang in their pre-triangulated counterexample and by showing that the resulting pre-$9$-angulation fails the higher mapping-cone axiom.
\\[0.5cm]
\medskip
\end{abstract}

\pagestyle{myheadings}
\markboth{\scriptsize J. He, J. He and P. Zhou}
         {\scriptsize A pre-$(n+2)$-angulated category not $(n+2)$-angulated}

\section{Introduction}

Triangulated categories were introduced in the 1960s by Grothendieck, Verdier and Puppe as an axiomatic framework for derived categories and stable homotopy theory \cite{H,P,V}. Geiss, Keller and Oppermann \cite{GKO} introduced $(n+2)$-angulated categories as higher-dimensional analogues of triangulated categories. An $(n+2)$-angulated category is an additive category equipped with an automorphism and a distinguished class of $(n+2)$-angles satisfying axioms (N1)--(N4). If the higher mapping-cone axiom (N4) is omitted, one obtains a pre-$(n+2)$-angulated category. Under axioms (N1)--(N3), the higher mapping-cone axiom is equivalent to the higher octahedral axiom introduced by Bergh and Thaule \cite{bt2}.

A natural question is whether every pre-$(n+2)$-angulated category must already be $(n+2)$-angulated. For a long time no counterexample was known, even in the classical case $n=1$; see, for example, \cite[Remark 2.1.5]{JM}. Chen, Liu, Lu and Zhang \cite{CLLZ} recently settled the classical case by constructing a Heller pre-triangulation on the category of finitely generated projective modules over the type $A_5$ preprojective algebra over $\mathbb F_2$ which fails Verdier's octahedral axiom.

The purpose of the present note is to show that the same phenomenon occurs in a genuinely higher dimension. The point is not simply to repeat the classical construction. The twisted complete comparison of \cite{CLLZ} is an involutive perturbation of the standard one, so an even power loses the obstruction. Taking the third power retains exactly one copy of the twist. We then combine this observation with the fixed-boundary rigidity in \cite{CLLZ} and Cheng--Yang's form of the higher mapping-cone axiom \cite{CY}.
\vspace{2mm}

Our main result is the following.

\begin{theorem}\label{HHZ}\rm
There exists an integer $n>1$ and a pre-$(n+2)$-angulated category which is not $(n+2)$-angulated. More precisely, let $n=7$, $k=\mathbb F_2$, $\Lambda=\Pi(A_5)$, $\mathcal F=\operatorname{proj}\text{-}\Lambda$, and let $\Sigma=(-)_\nu$, where $\nu$ is the graph-reflection automorphism of $\Lambda$ (equivalently, the automorphism defining the Nakayama twist). Then there is a class $\mathcal N_{\epsilon}^{[3]}$ of $9$-$\Sigma^3$-sequences such that
\[
(\mathcal F,\Sigma^3,\mathcal N_{\epsilon}^{[3]})
\]
is pre-$9$-angulated but not $9$-angulated.
\end{theorem}

This article is organised as follows. In Section 2, we recall the parametrisation of a pre-$(n+2)$-angulation due to Geiss--Keller--Oppermann and establish two compatibility lemmas needed later. In Section 3, we collect the ingredients from \cite{CLLZ} that will be used in the proof of our main results. In Section 4, we prove Proposition \ref{pre1}, which constructs a pre-$9$-angulation from the third power of a complete comparison. In Section 5, we prove Theorem \ref{th1} by showing that the resulting pre-$9$-angulation does not satisfy the higher mapping-cone axiom (N4).

\section{Preliminaries}

We first recall some basic definitions and results of (pre)$(n+2)$-angulated categories that will be
used throughout the paper.

Let $\C$ be an additive category equipped with an automorphism
$\Sigma\colon\C\rightarrow\C$, where $n$ is a positive integer.
An $(n+2)$-$\Sigma$-sequence in $\C$ is a sequence of objects and
morphisms
$$
A_0\xrightarrow{f_0}A_1\xrightarrow{f_1}A_2\xrightarrow{f_2}
\cdots\xrightarrow{f_{n-1}}A_n\xrightarrow{f_n}A_{n+1}
\xrightarrow{f_{n+1}}\Sigma A_0.
$$
Its \emph{left rotation} is the $(n+2)$-$\Sigma$-sequence
$$
A_1\xrightarrow{f_1}A_2\xrightarrow{f_2}A_3\xrightarrow{f_3}
\cdots\xrightarrow{f_n}A_{n+1}\xrightarrow{f_{n+1}}\Sigma A_0
\xrightarrow{(-1)^n\Sigma f_0}\Sigma A_1.
$$
The sequence is exact if the induced sequence
$$
\cdots\xrightarrow{}\C(-,A_0)\xrightarrow{}\C(-,A_1)\xrightarrow{}
\cdots\xrightarrow{}\C(-,A_{n+1})
\xrightarrow{}\C(-,\Sigma A_0)\xrightarrow{}
\cdots
$$
of representable functors $\C^{\mathrm{op}}\rightarrow\operatorname{Ab}$ is exact.

A \emph{morphism} between two $(n+2)$-$\Sigma$-sequences is a
collection of morphisms
$\varphi=(\varphi_0,\varphi_1,\ldots,\varphi_{n+1})$
such that the following diagram commutes:
$$
\xymatrix{
A_0 \ar[r]^{f_0}\ar[d]^{\varphi_0}
& A_1 \ar[r]^{f_1}\ar[d]^{\varphi_1}
& A_2 \ar[r]^{f_2}\ar[d]^{\varphi_2}
& \cdots \ar[r]^{f_n}
& A_{n+1} \ar[r]^{f_{n+1}}\ar[d]^{\varphi_{n+1}}
& \Sigma A_0 \ar[d]^{\Sigma\varphi_0}\\
B_0 \ar[r]^{g_0}
& B_1 \ar[r]^{g_1}
& B_2 \ar[r]^{g_2}
& \cdots \ar[r]^{g_n}
& B_{n+1} \ar[r]^{g_{n+1}}
& \Sigma B_0.
}
$$
Here each row is an $(n+2)$-$\Sigma$-sequence. Such a morphism is
called an \emph{isomorphism} if each
$\varphi_i$, $0\leqslant i\leqslant n+1$, is an isomorphism in $\C$.

\begin{definition}\cite[Definition 2.1]{GKO}\label{orig}
A \emph{pre-$(n+2)$-angulated category} is a triple
$(\C,\Sigma,\N)$, where $\C$ is an additive category,
$\Sigma$ is an automorphism of $\C$, called the $n$-suspension
functor, and $\N$ is a class of $(n+2)$-$\Sigma$-sequences,
whose elements are called $(n+2)$-angles, satisfying the following
axioms.

\begin{itemize}

\item[\textbf{(N1)}]
\begin{itemize}

\item[(a)]
The class $\N$ is closed under isomorphisms, direct sums and direct
summands.

\item[(b)]
For each object $A\in\C$, the trivial sequence
$$
A\xrightarrow{1_A}A\rightarrow0\rightarrow0\rightarrow
\cdots\rightarrow0\rightarrow\Sigma A
$$
belongs to $\N$.

\item[(c)]
Every morphism $f_0:A_0\rightarrow A_1$ in $\C$ can be extended to
an $(n+2)$-$\Sigma$-sequence
$$
A_0\xrightarrow{f_0}A_1\xrightarrow{f_1}A_2\xrightarrow{f_2}
\cdots\xrightarrow{f_{n-1}}A_n\xrightarrow{f_n}A_{n+1}
\xrightarrow{f_{n+1}}\Sigma A_0
$$
belonging to $\N$.

\end{itemize}

\item[\textbf{(N2)}]
An $(n+2)$-$\Sigma$-sequence belongs to $\N$ if and only if
its left rotation belongs to $\N$.

\item[\textbf{(N3)}]
Given a commutative diagram
$$
\xymatrix{
A_0 \ar[r]^{f_0}\ar[d]^{\varphi_0}
& A_1 \ar[r]^{f_1}\ar[d]^{\varphi_1}
& A_2 \ar[r]^{f_2}\ar@{-->}[d]^{\varphi_2}
& \cdots \ar[r]^{f_n}
& A_{n+1} \ar[r]^{f_{n+1}}\ar@{-->}[d]^{\varphi_{n+1}}
& \Sigma A_0 \ar[d]^{\Sigma\varphi_0}\\
B_0 \ar[r]^{g_0}
& B_1 \ar[r]^{g_1}
& B_2 \ar[r]^{g_2}
& \cdots \ar[r]^{g_n}
& B_{n+1} \ar[r]^{g_{n+1}}
& \Sigma B_0
}
$$
whose rows belong to $\N$, the dotted morphisms exist and complete
the diagram to a morphism of $(n+2)$-$\Sigma$-sequences.
\end{itemize}

In this case, the collection $\N$ is a pre-$(n+2)$-angulation of the category $\C$ (relative to the automorphism $\Sigma$). If the triplet $(\C,\Sigma,\N)$ moreover satisfies the following higher mapping-cone axiom, then it is called
an \emph{$(n+2)$-angulated category}.
\begin{itemize}
\item[\textbf{(N4)}]
In the situation of \textbf{(N3)}, the morphisms
$\varphi_2,\varphi_3,\ldots,\varphi_{n+1}$ can be chosen such that
the mapping cone
$$\hspace{-7mm}
A_1\oplus B_0
\xrightarrow{
\left(
\begin{smallmatrix}
-f_1&0\\
\varphi_1&g_0
\end{smallmatrix}
\right)}
A_2\oplus B_1
\xrightarrow{
\left(
\begin{smallmatrix}
-f_2&0\\
\varphi_2&g_1
\end{smallmatrix}
\right)}
\cdots
\xrightarrow{
\left(
\begin{smallmatrix}
-f_{n+1}&0\\
\varphi_{n+1}&g_n
\end{smallmatrix}
\right)}
\Sigma A_0\oplus B_{n+1}
\xrightarrow{
\left(
\begin{smallmatrix}
-\Sigma f_0&0\\
\Sigma\varphi_1&g_{n+1}
\end{smallmatrix}
\right)}
\Sigma A_1\oplus\Sigma B_0
$$
belongs to $\N$.

\end{itemize}
\end{definition}
\begin{remark}
Axiom (N4) is the original Geiss--Keller--Oppermann  higher mapping-cone axiom. Its equivalence with the higher octahedral axiom is proved by Bergh and Thaule \cite{bt2}. Cheng and Yang \cite{CY} introduced several further equivalent forms of (N4). We shall use their Axiom A in the following form.
\end{remark}

\begin{lemma}\rm\cite[Axiom $A$]{CY}\label{AX}
Let $(\C,\Sigma,\N)$ be a pre-$(n+2)$-angulated category and assume that it satisfies (N4).  Then the following commutative diagram with rows in $\N$
$$
\xymatrix{
A_0 \ar[r]^{f_0}\ar[d]^{\varphi_0}
& A_1 \ar[r]^{f_1}\ar[d]^{\varphi_1}
& \cdots \ar[r]^{f_{i-1}}
& A_i \ar[r]^{f_i}\ar[d]^{\varphi_i}
& A_{i+1} \ar[r]^{f_{i+1}}\ar@{-->}[d]^{\varphi_{i+1}}& \cdots \ar[r]^{f_n}& A_{n+1} \ar[r]^{f_{n+1}}\ar@{-->}[d]^{\varphi_{n+1}}
& \Sigma A_0 \ar[d]^{\Sigma\varphi_0}\\
B_0 \ar[r]^{g_0}
& B_1 \ar[r]^{g_1}
&  \cdots\ar[r]^{g_{i-1}}
& B_i  \ar[r]^{g_i}& B_{i+1}  \ar[r]^{g_{i+1}}& \cdots \ar[r]^{g_n}
& B_{n+1} \ar[r]^{g_{n+1}}
& \Sigma B_0
}
$$
for $1\leq i \leq n $ can be completed to a morphism of $(n+2)$-angles such that the mapping cone
$$
A_1\oplus B_0
\xrightarrow{
\left(
\begin{smallmatrix}
-f_1&0\\
\varphi_1&g_0
\end{smallmatrix}
\right)}
A_2\oplus B_1
\xrightarrow{
\left(
\begin{smallmatrix}
-f_2&0\\
\varphi_2&g_1
\end{smallmatrix}
\right)}
\cdots
\xrightarrow{
\left(
\begin{smallmatrix}
-f_{n+1}&0\\
\varphi_{n+1}&g_n
\end{smallmatrix}
\right)}
\Sigma A_0\oplus B_{n+1}
\xrightarrow{
\left(
\begin{smallmatrix}
-\Sigma f_0&0\\
\Sigma\varphi_1&g_{n+1}
\end{smallmatrix}
\right)}
\Sigma A_1\oplus\Sigma B_0
$$
belongs to $\N$.

\end{lemma}

\hspace{-5mm}{\bf 2.1 Complete comparisons and the Geiss--Keller--Oppermann parametrisation.}
~Let $\mathcal F$ be an additive category with split idempotents such that $\operatorname{mod}\mathcal F$, the category of finitely presented contravariant functors $\mathcal F^{\mathrm{op}}\to\operatorname{Ab}$, is Frobenius. We identify $\mathcal F$ with the full subcategory of projective-injective objects in $\operatorname{mod}\mathcal F$ via the Yoneda embedding. Write
$S=\Omega^{-1}$
for the suspension of the stable category $\underline{\operatorname{mod}}\mathcal F$.

Let $\Sigma$ be an automorphism of $\mathcal F$. It induces an exact automorphism, still denoted by $\Sigma$, of $\operatorname{mod}\mathcal F$ and hence a triangle automorphism of $\underline{\operatorname{mod}}\mathcal F$. Let
\[
X_\bullet:\quad X_0\xrightarrow{\alpha_0}X_1\xrightarrow{\alpha_1}\cdots\xrightarrow{\alpha_n}X_{n+1}\xrightarrow{\alpha_{n+1}}\Sigma X_0
\]
be an exact $(n+2)$-$\Sigma$-sequence in $\mathcal F$. Its \emph{first kernel} is
\[
K_X=\ker\bigl(\mathcal F(-,\alpha_0)\bigr)\in\operatorname{mod}\mathcal F.
\]
Following the complete comparison construction of \cite[Section 2]{GKO}, the sequence of representable functors determined by $X_\bullet$ is a finite part of a complete injective resolution of $K_X$. Comparing it with a fixed complete injective resolution and lifting the identity of $K_X$ gives a stable isomorphism
\begin{equation*}
\delta_{X_\bullet}:\Sigma K_X\xrightarrow{\sim}S^{n+2}K_X,
\end{equation*}
called the \emph{complete comparison} of $X_\bullet$. The comparison is independent of the auxiliary choices in the stable category.

We use the following form of the Heller--Geiss--Keller--Oppermann parametrisation. In the triangle-functor structure on $S^{n+2}$ one includes the usual sign $(-1)^{n+2}$ from \cite{GKO} over $\mathbb F_2$, which is the only field used in Sections 3--5, this sign is invisible.

\begin{theorem}\rm\cite[Lemma 2.3 and Proposition 2.4]{GKO}\label{mai}
Let
\[
\theta:\Sigma\xrightarrow{\sim}S^{n+2}
\]
be an isomorphism of triangle functors. Then the class $\mathcal N_\theta$ of all exact $(n+2)$-$\Sigma$-sequences $X_\bullet$ satisfying
$
\delta_{X_\bullet}=\theta_{K_X}
$
is a pre-$(n+2)$-angulation of $(\mathcal F,\Sigma)$. Conversely, every pre-$(n+2)$-angulation on $(\mathcal F,\Sigma)$ is obtained uniquely in this way.
\end{theorem}

The next elementary lemma replaces a direct-sum assertion which is generally false for a merely degreewise split sequence of complexes. We use the terminology of \cite[Section 1.3]{GKO}: an \emph{$(n+2)$-$\Gamma$-periodic complex} is a complex $X_\bullet$ together with the periodic identifications
\[
X_{i+n+2}=\Gamma X_i,~d_{i+n+2}=\Gamma(d_i)
\]
for all $i$, and morphisms and homotopies are required to respect these identifications.

\begin{lemma}\label{333}
Let $\Gamma$ be an automorphism of $\mathcal F$ inducing an exact automorphism of $\operatorname{mod}\mathcal F$, and let
\begin{equation}\label{eq:split-complexes}
0\longrightarrow X_\bullet\xrightarrow{x_\bullet}Y_\bullet\xrightarrow{q_\bullet}Z_\bullet\longrightarrow0
\end{equation}
be a degreewise split short exact sequence of exact $(n+2)$-$\Gamma$-periodic complexes of projective-injective objects. Use the mapping-cone convention of \textbf{(N4)}, so that the degree-$i$ term is $X_{i+1}\oplus Y_i$. Then the canonical chain map
\[
p\colon \operatorname{Cone}(x_\bullet)\longrightarrow Z_\bullet,
~p_i=(0,q_i),
\]
is a periodic homotopy equivalence. If $K_C$ and $K_Z$ denote the first kernels of $\operatorname{Cone}(x_\bullet)$ and $Z_\bullet$, respectively, and
$
u\colon K_C\xrightarrow{\sim}K_Z
$
is the stable isomorphism induced by $p$, then their complete comparisons satisfy
\begin{equation}\label{eq:comparison-naturality}
S^{n+2}(u)\circ\delta_{\operatorname{Cone}(x_\bullet)}
=\delta_{Z_\bullet}\circ\Gamma(u)
\end{equation}
in $\underline{\operatorname{mod}}\mathcal F$.
\end{lemma}

\begin{proof}
The kernel of $p$ is canonically isomorphic, as a complex, to $\operatorname{Cone}(1_{X_\bullet})$ in the same convention. Indeed, in degree $i$ one has
\[
\ker p_i=X_{i+1}\oplus x_i(X_i)\cong X_{i+1}\oplus X_i,
\]
and, under this identification, the restricted differential is
\[
(u,v)\longmapsto(-d_Xu,\,u+d_Xv),
\]
which is precisely the differential of the cone of the identity. Thus \eqref{eq:split-complexes} induces a degreewise split short exact sequence
\[
0\longrightarrow\operatorname{Cone}(1_{X_\bullet})
\longrightarrow\operatorname{Cone}(x_\bullet)
\xrightarrow{p}Z_\bullet\longrightarrow0.
\]
The complex $\operatorname{Cone}(1_{X_\bullet})$ is contractible, and its standard contracting homotopy is compatible with the $(n+2)$-$\Gamma$-periodicity. Since \eqref{eq:split-complexes} is degreewise split, splittings may be chosen on one period and extended by $\Gamma$. Hence the preceding short exact sequence yields a triangle in the homotopy category of $(n+2)$-$\Gamma$-periodic complexes, and $p$ is an isomorphism there. Thus $p$ is a periodic homotopy equivalence; compare \cite[Section 1.3]{GKO}.

It remains to record carefully the effect on complete comparisons. The chain map $p$ induces a morphism $u:K_C\to K_Z$ on first kernels. Since a homotopy inverse of $p$ induces an inverse to $u$ in the stable category, $u$ is a stable isomorphism. Choose complete injective resolutions $I_C$ and $I_Z$ of $K_C$ and $K_Z$. By the comparison theorem, $u$ lifts to a chain map
\[
U:I_C\longrightarrow I_Z,
\]
unique up to homotopy. The $(n+2)$-$\Gamma$-periodicity of $p$ means that, after one period, the induced map on the corresponding first kernels is $\Gamma(u)$. On the other hand, shifting the chosen lift by $n+2$ degrees induces $S^{n+2}(u)$ on the $(n+2)$nd cosyzygies. Consequently the two composites
\[
\Gamma K_C\longrightarrow S^{n+2}K_Z
\]
obtained respectively by first applying the complete comparison of $\operatorname{Cone}(x_\bullet)$ and then $S^{n+2}(u)$, or by first applying $\Gamma(u)$ and then the complete comparison of $Z_\bullet$, are represented by comparison maps lifting the same morphism through complete injective resolutions. Such lifts are unique up to homotopy. Therefore the square
\[
\xymatrix@C=44pt{
\Gamma K_C \ar[r]^{\delta_{\operatorname{Cone}(x_\bullet)}} \ar[d]_{\Gamma u}
& S^{n+2}K_C \ar[d]^{S^{n+2}u}\\
\Gamma K_Z \ar[r]_{\delta_{Z_\bullet}}
& S^{n+2}K_Z
}
\]
commutes in $\underline{\operatorname{mod}}\mathcal F$, which is exactly \eqref{eq:comparison-naturality}.
\end{proof}

\hspace{-5mm}{\bf 2.2~ The third power of a complete comparison.}~
We now consider the triangulated case, corresponding to $n=1$. Let
\[
\theta\colon \Sigma\xrightarrow{\sim}S^3
\]
be an isomorphism of triangle functors. Define a natural isomorphism
\begin{equation}\label{eq:third-power}
(\Theta_\theta^{[3]})_X
=S^6(\theta_X)\circ S^3(\theta_{\Sigma X})\circ\theta_{\Sigma^2X}\colon
\Sigma^3X\xrightarrow{\sim}S^9X.
\end{equation}
Since it is a composite of isomorphisms of triangle functors and their iterates, $\Theta_\theta^{[3]}\colon\Sigma^3\xrightarrow{\sim}S^9$ is again an isomorphism of triangle functors. Hence Theorem \ref{mai} determines a pre-$9$-angulation on $(\mathcal F,\Sigma^3)$.

If
\[
T:\quad X_0\longrightarrow X_1\longrightarrow X_2\longrightarrow\Sigma X_0
\]
is an exact $3$-$\Sigma$-sequence, write
\begin{equation*}
\begin{split}
P_3(T):\quad
X_0&\longrightarrow X_1\longrightarrow X_2\longrightarrow\Sigma X_0
\longrightarrow\Sigma X_1\longrightarrow\Sigma X_2\\
&\longrightarrow\Sigma^2X_0\longrightarrow\Sigma^2X_1\longrightarrow\Sigma^2X_2
\longrightarrow\Sigma^3X_0
\end{split}
\end{equation*}
for its threefold concatenation.

\begin{lemma}\label{lem:concatenation}
Assume that $\operatorname{mod}\mathcal F$ is $\mathbb F_2$-linear. If $T\in\mathcal N_\theta$, then $P_3(T)$ belongs to the pre-$9$-angulation determined by $\Theta_\theta^{[3]}$. More precisely, if $K$ is the first kernel of $T$, then the first kernel of $P_3(T)$ is $K$ and
\[
\delta_{P_3(T)}=(\Theta_\theta^{[3]})_K.
\]
\end{lemma}

\begin{proof}
Because $\Sigma$ induces an exact automorphism of $\operatorname{mod}\mathcal F$, the first kernels of $T$, $\Sigma T$, and $\Sigma^2T$ are naturally identified with
$
K, \Sigma K,\Sigma^2K,
$
respectively. The exact sequence of representable functors associated with $P_3(T)$ is obtained by splicing the resolution segment for $T$ with the corresponding segments for $\Sigma T$ and $\Sigma^2T$. Hence the first kernel of the concatenated sequence is still $K$.

Since $T\in\mathcal N_\theta$, its complete comparison is $\theta_K$. Moreover, $\mathcal N_\theta$ is a pre-$3$-angulation, so axiom \textbf{(N2)} applied three times shows that the third left rotation of $T$ again belongs to $\mathcal N_\theta$. Over $\mathbb F_2$ all rotation signs are trivial, and this third rotation is exactly the translated sequence $\Sigma T$. Repeating the same argument gives $\Sigma^2T\in\mathcal N_\theta$. Consequently, under the natural identifications of their first kernels, the complete comparisons of the two translated blocks are $\theta_{\Sigma K}$ and $\theta_{\Sigma^2K}$. Here we use the triangle-functor structures in the Heller--Geiss--Keller--Oppermann parametrisation. In general their iterates carry the usual signs coming from the suspension; over $\mathbb F_2$ these signs are trivial, so no additional sign occurs in the splicing below. Reading the spliced complete resolution from right to left, the resulting nine-step comparison is therefore
\[
\Sigma^3K
\xrightarrow{\ \theta_{\Sigma^2K}\ }
S^3\Sigma^2K
\xrightarrow{\ S^3(\theta_{\Sigma K})\ }
S^6\Sigma K
\xrightarrow{\ S^6(\theta_K)\ }
S^9K.
\]
The order of the three factors is forced by the order of the three concatenated blocks. By uniqueness of comparison morphisms up to homotopy, this composite is precisely the complete comparison of $P_3(T)$. Thus
\[
\delta_{P_3(T)}
=S^6(\theta_K)\circ S^3(\theta_{\Sigma K})\circ\theta_{\Sigma^2K}
=(\Theta_\theta^{[3]})_K,
\]
as required.
\end{proof}

The following observation will be the mechanism which preserves the obstruction of \cite{CLLZ} under the third power.

\begin{lemma}\label{lem:odd-twist}
Let $\epsilon:\operatorname{Id}\xrightarrow{\sim}\operatorname{Id}$ be a natural automorphism of $\underline{\operatorname{mod}}\mathcal F$ which commutes with $S$ and satisfies $\epsilon^2=1$. Let $\theta_0:\Sigma\xrightarrow{\sim}S^3$ be an isomorphism of triangle functors and set
\[
(\theta_\epsilon)_X=\epsilon_{S^3X}\circ(\theta_0)_X.
\]
Let $\Theta_0^{[3]}$ and $\Theta_\epsilon^{[3]}$ be the natural isomorphisms defined by \eqref{eq:third-power}. Then
\begin{equation}\label{eq:odd-twist}
(\Theta_\epsilon^{[3]})_X
=\epsilon_{S^9X}\circ(\Theta_0^{[3]})_X.
\end{equation}
Equivalently, if
$
(\lambda_*^{[3]})_X
:=\Omega^9\bigl((\Theta_*^{[3]})_X\bigr):
\Omega^9\Sigma^3X\longrightarrow X,
~ *=0,\epsilon,
$
then
\begin{equation}\label{eq:shifted-odd-twist}
(\lambda_\epsilon^{[3]})_X
=\epsilon_X\circ(\lambda_0^{[3]})_X.
\end{equation}
\end{lemma}

\begin{proof}
Since $\epsilon$ commutes with $S$, expanding the definition gives
\[
(\Theta_\epsilon^{[3]})_X
=\epsilon_{S^9X}S^6((\theta_0)_X)\,
\epsilon_{S^6\Sigma X}S^3((\theta_0)_{\Sigma X})\,
\epsilon_{S^3\Sigma^2X}(\theta_0)_{\Sigma^2X}.
\]
For brevity put
\[
f_1=S^6((\theta_0)_X),~
f_2=S^3((\theta_0)_{\Sigma X}),~
f_3=(\theta_0)_{\Sigma^2X}.
\]
Naturality of $\epsilon$ gives
\[
\epsilon_{S^9X}f_1=f_1\epsilon_{S^6\Sigma X},
~
\epsilon_{S^6\Sigma X}f_2=f_2\epsilon_{S^3\Sigma^2X}.
\]
Consequently,
\[
\begin{aligned}
&\epsilon_{S^9X}f_1\epsilon_{S^6\Sigma X}f_2
\epsilon_{S^3\Sigma^2X}f_3\\
&\qquad =f_1f_2\epsilon_{S^3\Sigma^2X}f_3
=f_1\epsilon_{S^6\Sigma X}f_2f_3
=\epsilon_{S^9X}f_1f_2f_3,
\end{aligned}
\]
where the first equality uses $\epsilon^2=1$. This is \eqref{eq:odd-twist}. Applying $\Omega^9$ and using the commutation of $\epsilon$ with $\Omega$ yields \eqref{eq:shifted-odd-twist}.
\end{proof}

\section{The Chen--Liu--Lu--Zhang's construction}

In this section we record the precise ingredients from \cite{CLLZ} that will be used in the higher construction. Throughout this section, $k=\mathbb F_2$, all $\Lambda$-modules are finite-dimensional right modules, and
\[
\mathcal F=\operatorname{proj}\text{-}\Lambda.
\]
By \cite[Proposition 3.1]{CLLZ}, the algebra $\Lambda$ is self-injective. Hence $\operatorname{mod}\Lambda$ is a Frobenius category, and its stable category $\underline{\operatorname{mod}}\Lambda$ is triangulated with suspension $S=\Omega^{-1}$ by \cite[Theorem 2.6]{H}. Under the usual equivalence $\operatorname{mod}\mathcal F\simeq\operatorname{mod}\Lambda$, this is the stable category appearing in Section 2.

Let $Q$ be the linearly oriented quiver of type $A_5$,
\[
1\longrightarrow2\longrightarrow3\longrightarrow4\longrightarrow5,
\]
and let $\overline Q$ be its double. The type $A_5$ preprojective algebra is
\[
\Lambda=\Pi(A_5)=k\overline Q\Big/\left(\sum_{a\in Q_1}(aa^*-a^*a)\right).
\]
Let $\nu$ be the graph-reflection automorphism of \cite[Section 3.1]{CLLZ}; it has order two and realizes the Nakayama twist. It induces an automorphism
\[
\Sigma=(-)_\nu:\mathcal F\longrightarrow\mathcal F.
\]
By \cite[Remark 3.3]{CLLZ}, Chen--Liu--Lu--Zhang constructed a standard triangulation $\Delta_0$ on $(\mathcal F,\Sigma)$ with complete comparison
\begin{equation*}
\theta_0:\Sigma\xrightarrow{\sim}S^3.
\end{equation*}

We recall from \cite[Section 3.2]{CLLZ} the module which supports their twist. Let $M\in\operatorname{mod}\Lambda$ have dimension vector $(2,2,2,1,0)$ and arrow matrices
\[
x_1=\begin{bmatrix}0&1\\0&1\end{bmatrix},~
y_1=\begin{bmatrix}1&1\\0&0\end{bmatrix},~
x_2=\begin{bmatrix}1&0\\0&0\end{bmatrix},
\]
\[
y_2=\begin{bmatrix}0&1\\0&0\end{bmatrix},~
x_3=\begin{bmatrix}0&1\end{bmatrix},~
y_3=\begin{bmatrix}1\\0\end{bmatrix},
\]
with all maps at the fifth vertex equal to zero.

\begin{proposition}\rm\cite[Proposition 3.6]{CLLZ}\label{prop:M}
The module $M$ is indecomposable and
\[
\Omega M\cong M_\nu,~\Omega^2M\cong M,
\]
\begin{equation}\label{eq:endM}
\underline{\operatorname{End}}_\Lambda(M)\cong k[\rho]/(\rho^2),
\end{equation}
where $0\neq\rho$ is the unique nonzero radical stable endomorphism of $M$. Moreover,
\[
\underline{\operatorname{Hom}}_\Lambda(M,\Omega M)=0,
\qquad
\underline{\operatorname{Hom}}_\Lambda(\Omega M,M)=0.
\]
\end{proposition}

Put $\mathcal O=\{M,\Omega M\}$. By \cite[Proposition 3.7]{CLLZ} there is a natural transformation
\[
\eta:\operatorname{Id}_{\underline{\operatorname{mod}}\Lambda}
\longrightarrow
\operatorname{Id}_{\underline{\operatorname{mod}}\Lambda}
\]
supported on $\operatorname{add}\mathcal O$, with $\eta_M=\rho$, $\eta^2=0$, and commuting with $S$. Hence
$
\epsilon=1+\eta
$
is a natural automorphism commuting with $S$, and, because $k=\mathbb F_2$,
\begin{equation}\label{eq:epsilon-involution}
\epsilon^2=1.
\end{equation}
The twisted complete comparison is
\begin{equation}\label{eq:thetaepsilon}
(\theta_\epsilon)_N
=\epsilon_{S^3N}\circ(\theta_0)_N:
\Sigma N\xrightarrow{\sim}S^3N.
\end{equation}
By \cite[Proposition 3.7 and Corollary 3.8]{CLLZ}, $\theta_\epsilon$ is an isomorphism of triangle functors and determines a pre-triangulation $\Delta_\epsilon$ on $(\mathcal F,\Sigma)$.

The following fixed-boundary configuration collects the precise information from \cite{CLLZ} needed below. Properties (H1)--(H2) are supplied by their construction in Lemma 3.12, property (H3) uses the twist of Proposition 3.7 together with the orbit information built into that construction, and property (H4) is the fixed-boundary rigidity proved in Lemma 3.13. We include the first-kernel statement explicitly because it is used in Section 5.

\begin{lemma}\rm\cite[Propositions 3.7 and 3.11, Lemmas 3.12--3.13]{CLLZ}\label{lem:CLLZ}
There exist modules $A,B$ and exact $3$-$\Sigma$-sequences
\[
T_X:\quad E_{X,0}\xrightarrow{d_{X,0}}E_{X,1}
\xrightarrow{d_{X,1}}E_{X,2}
\xrightarrow{d_{X,2}}\Sigma E_{X,0},
~ X\in\{A,B,M\},
\]
with the following properties.
\begin{enumerate}
\item[(H1)] Under $\operatorname{mod}\mathcal F\simeq\operatorname{mod}\Lambda$, the first kernel of $T_X$ is $X$. There is a degreewise split short exact sequence of $\Sigma$-periodic complexes
\begin{equation}\label{eq:CLLZ-ses}
0\longrightarrow T_A\xrightarrow{x_\bullet}T_B\xrightarrow{q_\bullet}T_M\longrightarrow0.
\end{equation}
\item[(H2)] Each of $T_A,T_B,T_M$ has complete comparison $\theta_0$; equivalently, all three belong to $\Delta_0$.
\item[(H3)] Neither $A$ nor $B$ has a stable direct summand in $\mathcal O$. Consequently, the twist is trivial on their syzygy orbits and
$
T_A,T_B\in\Delta_\epsilon.
$
\item[(H4)] Keep $x_0$ and $x_1$ fixed. If
$x:E_{A,2}\to E_{B,2}$ completes them to a morphism $T_A\to T_B$, then
$
x=sx_2
$
for an automorphism $s$ of $E_{B,2}$ satisfying
$
sd_{B,1}=d_{B,1}, d_{B,2}s=d_{B,2}.
$
Equivalently, $(1,1,s)$ is a strict automorphism of the periodic sequence $T_B$ fixing its first two components.
\end{enumerate}
\end{lemma}

\section{Construction of the pre-$9$-angulation}

Apply the construction of Section 2.3 to the twisted comparison \eqref{eq:thetaepsilon}. Thus
\begin{equation*}
(\Theta_\epsilon^{[3]})_X
=S^6((\theta_\epsilon)_X)\circ
S^3((\theta_\epsilon)_{\Sigma X})\circ
(\theta_\epsilon)_{\Sigma^2X}:
\Sigma^3X\xrightarrow{\sim}S^9X.
\end{equation*}

\begin{construction}\label{cons}
Let $\mathcal N_\epsilon^{[3]}$ be the class of exact $9$-$\Sigma^3$-sequences $X_\bullet$ in $\mathcal F$ whose complete comparison satisfies
\[
\delta_{X_\bullet}=(\Theta_\epsilon^{[3]})_{K_X}.
\]
\end{construction}

\begin{proposition}\label{pre1}\rm
The triple $(\mathcal F,\Sigma^3,\mathcal N_\epsilon^{[3]})$ is a pre-$9$-angulated category.
\end{proposition}

\begin{proof}
The category $\mathcal F$ has split idempotents, and $\operatorname{mod}\mathcal F$ is Frobenius because $(\mathcal F,\Sigma,\Delta_0)$ is triangulated. The natural isomorphism $\Theta_\epsilon^{[3]}:\Sigma^3\to S^9$ is an isomorphism of triangle functors by Section 2.3. The assertion therefore follows from Theorem \ref{mai} with $n=7$.
\end{proof}

By (H3), $T_A,T_B\in\Delta_\epsilon$. Lemma \ref{lem:concatenation} therefore gives
\begin{equation}\label{eq:P3AB}
P_3(T_A),\;P_3(T_B)\in\mathcal N_\epsilon^{[3]}.
\end{equation}
The morphism $x_\bullet$ in \eqref{eq:CLLZ-ses} concatenates with its two $\Sigma$-translates to a morphism
\begin{equation}\label{eq:P3x}
P_3(x_\bullet):P_3(T_A)\longrightarrow P_3(T_B).
\end{equation}

For later use, let $\Theta_0^{[3]}$ be the third power of $\theta_0$ and put
\[
(\lambda_*^{[3]})_X
=\Omega^9\bigl((\Theta_*^{[3]})_X\bigr):
\Omega^9\Sigma^3X\longrightarrow X,
~*=0,\epsilon.
\]
By Lemma \ref{lem:odd-twist}, \eqref{eq:epsilon-involution}, and $\epsilon_M=1_M+\rho$, one has the well-typed identity
\begin{equation}\label{eq:obstruction}
(\lambda_\epsilon^{[3]})_M
=(1_M+\rho)\circ(\lambda_0^{[3]})_M.
\end{equation}
This is the point at which the odd power is essential: an even power would contain an even number of $\epsilon$-factors and the involution \eqref{eq:epsilon-involution} would remove the twist.

\section{Failure of the higher mapping-cone axiom}

\begin{theorem}\label{th1}\rm
The pre-$9$-angulated category $(\mathcal F,\Sigma^3,\mathcal N_\epsilon^{[3]})$ does not satisfy axiom \textbf{(N4)}.
\end{theorem}

\begin{proof}
Assume for contradiction that \textbf{(N4)} holds. Write the nine components of the morphism \eqref{eq:P3x}, using indices $0,\ldots,8$, as
\[
x_0,x_1,x_2,\Sigma x_0,\Sigma x_1,\Sigma x_2,
\Sigma^2x_0,\Sigma^2x_1,\Sigma^2x_2.
\]
By \eqref{eq:P3AB}, both rows lie in $\mathcal N_\epsilon^{[3]}$. Apply Lemma \ref{AX} with $n=7$ and $i=7$ to the first eight components. We obtain a morphism of $9$-angles
\[
\phi\colon P_3(T_A)\longrightarrow P_3(T_B)
\]
such that
\begin{equation}\label{eq:phi-fixed}
\phi_j=P_3(x_\bullet)_j~(0\leq j\leq7),
~
\operatorname{Cone}(\phi)\in\mathcal N_\epsilon^{[3]}.
\end{equation}
The last three components $\phi_6,\phi_7,\phi_8$ form a morphism
$\Sigma^2T_A\to\Sigma^2T_B$. By \eqref{eq:phi-fixed}, its first two components are $\Sigma^2x_0$ and $\Sigma^2x_1$. Applying (H4) after $\Sigma^2$, there is an automorphism $s$ of $E_{B,2}$ as in (H4) such that
\[
\phi_8=(\Sigma^2s)(\Sigma^2x_2).
\]
Let $a$ be the automorphism of $P_3(T_B)$ which is the identity in degrees $0,1,\cdots,7$ and has component $\Sigma^2s^{-1}$ in degree $8$. The relations in (H4) imply that $a$ is indeed an automorphism of the $9$-$\Sigma^3$-sequence $P_3(T_B)$, and
$
a\phi=P_3(x_\bullet).
$
The mapping cones of $\phi$ and $a\phi$ are isomorphic by the induced block-diagonal change of coordinates. Since $\mathcal N_\epsilon^{[3]}$ is closed under isomorphisms, \eqref{eq:phi-fixed} yields
\begin{equation}\label{eq:cone-standard-in-N}
\operatorname{Cone}(P_3(x_\bullet))\in\mathcal N_\epsilon^{[3]}.
\end{equation}
Concatenating \eqref{eq:CLLZ-ses} gives a degreewise split short exact sequence of $9$-$\Sigma^3$-periodic complexes
\begin{equation*}
0\longrightarrow P_3(T_A)\xrightarrow{P_3(x_\bullet)}P_3(T_B)
\longrightarrow P_3(T_M)\longrightarrow0.
\end{equation*}
Set
\[
C_\bullet=\operatorname{Cone}(P_3(x_\bullet)),
~ K_C=K_{C_\bullet}.
\]
By Lemma \ref{333}, applied with $n=7$ and periodicity automorphism $\Gamma=\Sigma^3$, the canonical projection
\[
p\colon C_\bullet\longrightarrow P_3(T_M)
\]
is a periodic homotopy equivalence. Let
$
u\colon K_C\xrightarrow{\sim}M
$
be the induced stable isomorphism of first kernels, where we use (H1) to identify the first kernel of $P_3(T_M)$ with $M$. By (H2) and Lemma \ref{lem:concatenation},
\begin{equation*}
\delta_{P_3(T_M)}=(\Theta_0^{[3]})_M.
\end{equation*}
The naturality formula \eqref{eq:comparison-naturality} of Lemma \ref{333} therefore gives
\begin{equation}\label{eq:cone-comparison-square}
S^9(u)\circ\delta_{C_\bullet}
=(\Theta_0^{[3]})_M\circ\Sigma^3(u).
\end{equation}
On the other hand, \eqref{eq:cone-standard-in-N} and the definition of $\mathcal N_\epsilon^{[3]}$ give
\[
\delta_{C_\bullet}=(\Theta_\epsilon^{[3]})_{K_C}.
\]
Since $\Theta_\epsilon^{[3]}\colon \Sigma^3\to S^9$ is a natural transformation, naturality with respect to $u$ yields
\begin{equation}\label{eq:twisted-comparison-square}
S^9(u)\circ(\Theta_\epsilon^{[3]})_{K_C}
=(\Theta_\epsilon^{[3]})_M\circ\Sigma^3(u).
\end{equation}
Combining \eqref{eq:cone-comparison-square} and \eqref{eq:twisted-comparison-square}, and cancelling the isomorphism $\Sigma^3(u)$ in the stable category, we obtain
\[
(\Theta_\epsilon^{[3]})_M=(\Theta_0^{[3]})_M.
\]
Applying $\Omega^9$ and using \eqref{eq:obstruction} gives
\[
(1_M+\rho)\circ(\lambda_0^{[3]})_M=(\lambda_0^{[3]})_M.
\]
Since $(\lambda_0^{[3]})_M$ is an isomorphism, it follows that $\rho=0$, contradicting \eqref{eq:endM}. Therefore \textbf{(N4)} fails.
\end{proof}

\begin{proof}[\bf \emph{Proof of Theorem \ref{HHZ}}]~
Proposition \ref{pre1} gives the required pre-$9$-angulation, and Theorem \ref{th1} shows that it is not $9$-angulated.
\end{proof}
\vspace{2mm}

\paragraph{Competing Interests}
The authors declare that they have no conflicts of interest to this work.

\paragraph{Data Availability}
Data sharing not applicable to this article as no datasets were generated or analysed during the current study.

\textbf{Jian He}\\
Department of Applied Mathematics, Lanzhou University of Technology,\\
730050 Lanzhou, Gansu, P. R. China\\
E-mail: \textsf{jianhe30@163.com}\\[0.3cm]
\textbf{Jing He}\\
School of Mathematics and Statistics, Hunan University of Technology and Business, 410205 Changsha, Hunan P. R. China\\
E-mail: \textsf{jinghe1003@163.com}
\\[0.3cm]
\textbf{Panyue Zhou}\\
School of Mathematics and Statistics, Changsha University of Science and Technology, 410114 Changsha, Hunan,  P. R. China\\
E-mail: \textsf{panyuezhou@163.com}

\end{document}